\documentclass[12pt]{elsarticle}  

\usepackage{amssymb}
\usepackage{amsxtra}

\newcommand{\bC}{\mathbb{C}}
\newcommand{\bR}{\mathbb{R}}

\newcommand{\bN}{\mathbb{N}}

\newcommand{\re}{{\rm \, Re \,}}

 \newtheorem{theorem}{Theorem}[section]
 
 \newtheorem{lemma}[theorem]{Lemma}

 \newdefinition{example}[theorem]{Example}

 \newenvironment{proof}{{\bf Proof:} }

\begin{document}
\baselineskip 6.6mm  

\title{An infinite-dimensional generalization of Zenger's lemma\tnoteref{t}}
\tnotetext[t]{Please cite this article in press as: 
R. Drnov\v sek, An infinite-dimensional generalization of Zenger's lemma, 
J. Math. Anal. Appl. (2011), doi:10.1016/j.jmaa.2011.11.018 }

\author{Roman Drnov\v sek}
\ead{roman.drnovsek@fmf.uni-lj.si}
\address{Department of Mathematics,
    Faculty of Mathematics and Physics, \\
    University of Ljubljana, \\
    Jadranska 19,
    SI-1000 Ljubljana, Slovenia}

\begin{abstract}
We prove an infinite-dimensional generalization of Zenger's lemma that was used in the proof of the fact 
that the convex hull of the point spectrum of a linear operator is contained in its numerical range.
Two relevant examples are given, and possible application in the Arrow-Debreu model is also discussed.
\end {abstract}


\maketitle

\vspace{5mm}
\section{Introduction}

In 1968, Zenger \cite{Ze} proved the following result that is known as Zenger's lemma.
 
\begin{theorem}
\label{Zenger_finite} 
Let $\|\cdot\|$ be a norm on $\bC^n$, let $\alpha_k > 0$ for all $k = 1, \ldots, n$, and let $\sum_{k=1}^{n} \alpha_k = 1$. 
Then there exists a vector $w = (w_1, w_2, \ldots, w_n) \in \bC^n$ with $\|w\| = 1$ and $w_1 w_2 \cdots w_n \neq 0$ such that 
the functional $\phi$ on $\bC^n$ defined by 
$$ \phi(z) = \sum_{k=1}^{n} \frac{\alpha_k \, z_k}{w_k} \ \ \ ( z = (z_1, z_2, \ldots, z_n) \in \bC^n ) $$
has norm one.
\end{theorem}

Zenger applied Theorem \ref{Zenger_finite} in the proof of the fact that the convex hull of the point spectrum of 
a linear operator is contained in its numerical range; see also \cite[Section 19]{BoDu}.
Some years ago, Theorem \ref{Zenger_finite} was also applied in the theory of invariant subspaces; see \cite{AmMu} or \cite{ChPa}. 
The aim of this note is to extend the theorem from the space $\bC^n$ to the classical sequence space $l^\infty$.

We omit the proof of the following known lemma, since it can be easily proved by an application of the uniform boundedness principle.
We refer the reader to \cite{AlBu03} or \cite{AlBu06} for details concerning absolute weak topologies.

\begin{lemma}
\label{equivalent_topologies}
For a net $\{x^{(i)}\}_{i \in I}$ in $l^\infty$ and a vector $x \in l^\infty$, the following assertions are equivalent: 
  
(a) The net $\{x^{(i)}\}_{i \in I}$ converges to $x$ in the absolute weak topology $|\sigma| (l^\infty, l^1)$, i.e., for each $y \in l^1$, 
	   $$ \lim_{i \in I} \left( \sum_{k=1}^{\infty} |x^{(i)}_k - x_k| \cdot |y_k| \right) =  0 ;$$
	   
(b) The net $\{x^{(i)}\}_{i \in I}$ converges to $x$ in the weak* topology $\sigma (l^\infty, l^1)$, i.e., for each $y \in l^1$, 
	   $$ \lim_{i \in I} \left( \sum_{k=1}^{\infty} (x^{(i)}_k - x_k) \cdot y_k  \right) = 0 ;$$
	   
(c) For each $k \in \bN$, $\lim_{i \in I} x^{(i)}_k = x_k$, and the net $\{x^{(i)}\}_{i \in I}$ is norm bounded, i.e., 
    $\sup \{ \|x^{(i)}\|_{\infty} : i \in I \} < \infty$.
\end{lemma}

We will mention several times the order ideal $I_w$ generated by a vector $w = (w_1, w_2, \ldots) \in l^\infty$.
For this smallest order ideal containing $w$ it holds that 
$$ I_w = \{ x \in l^\infty : \textrm{ there exists } \lambda > 0 \textrm{  such that  } |x_k| \le \lambda |w_k| \textrm{  for all  } k \in \bN \} . $$
We refer the reader to \cite{AlBu03} or \cite{AlBu06} for the theory of Riesz spaces.

\vspace{5mm}
\section{The result}

Let $\|\cdot\|$ be a norm on the classical sequence space $l^\infty$ that is equivalent to the original norm $\|\cdot\|_{\infty}$, i.e., 
there are numbers $C \ge  c > 0$ such that 
\begin{equation}
c \, \|x\|_{\infty} \leq \|x\| \leq  C \, \|x\|_{\infty} \ \ \textrm{for all} \ x \in l^\infty . 
\label{equivalent_norms}
\end{equation}
By $\| \cdot \|_*$ we denote the predual norm to the norm $\|\cdot\|$, i.e., the norm of a vector $y = (y_1, y_2, \ldots) \in l^1$ is defined by  
$$ \| y \|_* = \sup \left\{ \left| \sum_{k=1}^{\infty} x_k \, y_k \right| : x \in l^\infty, \|x\| \le 1 \right\} . $$

For $N \in \bN$, let $P_N$ be the natural projection on $l^\infty$ defined by 
$$ P_N(x_1, x_2, \ldots) = (x_1, x_2, \ldots, x_N, 0, 0, \ldots) . $$
Since $P_N^2 = P_N$, the operator norm $\|P_N\|$ of $P_N$ with respect to the norm $\|\cdot\|$ is at least $1$.

Some properties of these projections are assumed in the following extension of 
Theorem \ref{Zenger_finite} from the space $\bC^n$ to the sequence space $l^\infty$.  
  
\begin{theorem}
\label{Zenger} 
Let $\|\cdot\|$ be a norm on $l^\infty$ that is equivalent to the norm $\|\cdot\|_{\infty}$, i.e., (\ref{equivalent_norms}) holds.
Suppose also that 
\begin{equation}
\liminf_{N \to \infty} \|P_N \| = 1 \ \ \ \textrm{and} \ \ \ \  \|x\| = \liminf_{N \to \infty} \|P_N x \| \ \ \textrm{for all} \ x \in l^\infty . 
\label{P_N}
\end{equation}
Let $\alpha = (\alpha_1, \alpha_2, \alpha_3, \ldots) \in l^1$ be a sequence of strictly positive numbers such that 
$\|\alpha\|_1 = \sum_{k=1}^{\infty} \alpha_k = 1$. Then there exist $w \in l^\infty$ and $\phi \in l^1$ such that 
$$  \| w \| = 1 \ , \ \  \| \phi \|_* = 1 \ \ \textrm{and} \ \ w_k \, \phi_k = \alpha_k \ \ \ \textrm{for all} \ k \in \mathbb{N}. $$
\end{theorem}

\begin{proof}
Clearly, we can assume that $c = 1$ in (\ref{equivalent_norms}). By the Banach-Alaoglu theorem, the unit ball 
$B_{\infty} = \{ x \in l^\infty : \|x\|_{\infty} \le 1 \}$ is compact in the weak* topology $\sigma(l^\infty, l^1)$. 
Since the unit ball $B = \{ x \in l^\infty : \|x\| \le 1 \}$ is contained in $B_{\infty}$, it is also compact 
in $\sigma(l^\infty, l^1)$ if we show that it is closed in $\sigma(l^\infty, l^1)$. 
To end this, pick an arbitrary net $\{x^{(i)}\}_{i \in I}$ in $B$ converging to a vector $x \in l^\infty$. 
By Lemma \ref{equivalent_topologies}, $\lim_{i \in I} x^{(i)}_k = x_k$ for each $k \in \bN$, and so $\lim_{i \in I} \| P_N (x^{(i)}- x)\|_{\infty} = 0$ for every $N \in \bN$. 
Since the norm $\|\cdot\|$ is equivalent to the norm $\|\cdot\|_{\infty}$, we obtain 
$\lim_{i \in I} \| P_N (x^{(i)} - x)\|= 0$ for every $N \in \bN$. 
Now, the inequality 
$$ \|P_N x\| \le \|P_N (x - x^{(i)})\| + \|P_N x^{(i)}\| \le \|P_N (x - x^{(i)})\| + \|P_N\| $$
implies that $\|P_N x\| \le \|P_N\|$, and so   
$$ \|x\| = \liminf_{N \to \infty} \|P_N x \| \le \liminf_{N \to \infty} \|P_N \| = 1 , $$
that is, $x \in B$ as desired.
 
Define a function $F : B \to [- \infty, \infty)$ by 
$$ F(x) =  \sum_{k=1}^{\infty} \alpha_k  \cdot \log |x_k| . $$
Since $\|x\|_{\infty} \leq  1$ for each $x \in B$, $\log |x_k| \le 0$ for each $k$, and so the series above converges in $[- \infty, 0]$.  
We claim that the function $F$ is upper semicontinuous in the topology $\sigma(l^\infty, l^1)$.
Pick any $c \in \mathbb{R}$ and any net $\{x^{(i)}\}_{i \in I}$ in $B$ converging to $x \in B$.
We must show that $F(x) \geq c$ if $F(x^{(i)}) \geq c$ for all $i \in I$. 
For each $m \in \bN$ we define $y^{(m)} \in l^\infty$ by $y^{(m)}_k = \max \{|x_k|, \frac{1}{m}\}$ $(k \in \bN)$. 
Observe that $F(y^{(m)}) \in (- \infty, 0]$, since $1 \geq  y^{(m)}_k  \geq \frac{1}{m}$.

By Jensen's inequality, we have, for every $m \in \bN$ and $i \in I$,  
$$ F(x^{(i)}) - F(y^{(m)}) =  \sum_{k=1}^{\infty} \alpha_k  \cdot \log \left( \frac{|x^{(i)}_k|}{y^{(m)}_k} \right) \leq 
   \log \left( \sum_{k=1}^{\infty} \alpha_k  \cdot \frac{|x^{(i)}_k|}{y^{(m)}_k} \right) , $$
and so 
\begin{equation}
F(y^{(m)}) +  \log \left( \sum_{k=1}^{\infty} \frac{\alpha_k}{y^{(m)}_k} \cdot |x^{(i)}_k| \right) \geq   F(x^{(i)}) \geq c . 
\label{eq1}
\end{equation}
Since $0 < \frac{\alpha_k}{y^{(m)}_k} \le m \cdot \alpha_k$ for all $k$ and since $\alpha \in l^1$, we have 
$(\frac{\alpha_1}{y^{(m)}_1}, \frac{\alpha_2}{y^{(m)}_2}, \ldots ) \in l^1$, and so Lemma \ref{equivalent_topologies} implies that 
$$ \sum_{k=1}^{\infty} \frac{\alpha_k}{y^{(m)}_k} \cdot |x^{(i)}_k| \xrightarrow{i \in I} 
   \sum_{k=1}^{\infty} \frac{\alpha_k}{y^{(m)}_k} \cdot |x_k| . $$
Consequently, we obtain from (\ref{eq1}) that
\begin{equation}
F(y^{(m)}) +  \log \left( \sum_{k=1}^{\infty} \frac{\alpha_k}{y^{(m)}_k} \cdot |x_k| \right) \geq  c . 
\label{eq2}
\end{equation}
Now, by the monotone convergence theorem,
$$  \sum_{k=1}^{\infty} \frac{\alpha_k}{y^{(m)}_k} \cdot |x_k| \xrightarrow{m \to \infty} \sum_{k=1}^{\infty} \alpha_k = 1 $$
and  
$$ - F(y^{(m)}) = - \sum_{k=1}^{\infty} \alpha_k  \cdot \log (y^{(m)}_k)  \xrightarrow{m \rightarrow \infty} 
- \sum_{k=1}^{\infty} \alpha_k  \cdot \log |x_k| = - F(x) , $$
and so we obtain from (\ref{eq2}) that  $F(x) \geq c$ as claimed.

Since every upper semicontinuous function attains its maximum on a compact set, there exists a vector $w \in B$ such that 
$F(x) \le F(w)$ for all $x \in B$. Since $F(x) =  - \log C$ for $x = 1/C \cdot (1, 1, 1, \ldots) \in B$, we have 
$0 \geq F(w) \geq - \log C$. Also, $\|w\| = 1$ and $w_k \neq 0$ for all $k$. 
Now, fix $x \in B$ such that, for some $\lambda > 0$, $|x_k| \le \lambda |w_k|$ for all $k \in \bN$, 
that is, $x$ belongs to the order ideal $I_w$ generated by $w$.   
Since $(1 - \frac{1}{m}) w + \frac{1}{m} x \in B$ for all $m \in \bN$, it holds that 
$m \, F((1 - \frac{1}{m}) w + \frac{1}{m} x) \le m \, F(w)$ which rewrites to the inequality
$$ 0 \ge  \sum_{k=1}^{\infty} \alpha_k  \cdot \log \left| \frac{(1 - \frac{1}{m}) w_k + \frac{1}{m} x_k}{w_k} \right|^m = $$
\begin{equation}
\label{negative}
 = \sum_{k=1}^{\infty} \alpha_k  \cdot \log \left| 1  + \frac{1}{m} \left( \frac{x_k}{w_k} - 1 \right)\right|^m \ge 
\end{equation}
$$ \ge \sum_{k=1}^{\infty} \alpha_k  \cdot \log \left| 1 + \frac{1}{m} \left( \re \left( \frac{x_k}{w_k} \right)- 1 \right) \right|^m , $$
as $|\re z| \leq |z|$ ($z \in \bC$).  
Now we will use the known fact that, for any $x \ge 0$, the sequence $\{(1+x/m)^m\}_{m \in \bN}$ increases to $e^x$, while 
the sequence $\{(1-x/m)^{-m}\}_{m \in \bN}$ decreases to $e^x$. 
Since 
$$ \left| \re \left( \frac{x_k}{w_k} \right) - 1 \right| \leq \left| \frac{x_k}{w_k} \right| + 1 \leq \lambda + 1 , $$
the sums of positive terms 
$$ \sum_{\re  \left( \frac{x_k}{w_k} \right) > 1} \alpha_k  \cdot \log \left( 1 + \frac{1}{m} \left( \re \left( \frac{x_k}{w_k} \right) - 1 \right) \right)^m $$
increase with $m$ to a finite limit 
$$ \sum_{\re \left(\frac{x_k}{w_k}\right) > 1} \alpha_k  \cdot \left( \re \left( \frac{x_k}{w_k} \right) - 1 \right)  $$
by the monotone convergence theorem, while the sums
$$ - \sum_{\re \left(\frac{x_k}{w_k}\right) \le 1} \alpha_k  \cdot \log \left| 1 + 
     \frac{1}{m} \left( \re \left( \frac{x_k}{w_k} \right) - 1 \right) \right|^m = $$
$$ =  \sum_{\re \left(\frac{x_k}{w_k}\right) \le 1} \alpha_k  \cdot \log \left| 1 - 
     \frac{1}{m} \left( 1 - \re \left(\frac{x_k}{w_k}\right) \right) \right|^{- m} $$
decrease with $m$ (provided $m > \lambda + 1$) to the sum
$$  \sum_{\re \left(\frac{x_k}{w_k}\right) \le 1} \alpha_k  \cdot \left( 1 - \re \left( \frac{x_k}{w_k} \right) \right) $$
by the dominated convergence theorem.  
Therefore, we conclude from (\ref{negative}) that 
$$ \sum_{k=1}^{\infty} \alpha_k  \cdot \left( \re \left( \frac{x_k}{w_k} \right) - 1 \right) \le 0 , $$
and so 
$$ \re \left( \sum_{k=1}^{\infty} \alpha_k  \cdot \frac{x_k}{w_k} \right) = 
\sum_{k=1}^{\infty} \alpha_k  \cdot \re \left( \frac{x_k}{w_k} \right) \le  \sum_{k=1}^{\infty} \alpha_k  = 1 . $$
Since we can replace $x$ by $e^{i t}x$ ($t \in \bR$), it must hold that 
$$ \left| \sum_{k=1}^{\infty} \alpha_k  \cdot \frac{x_k}{w_k} \right| \le  1 , $$
that is, 
\begin{equation}
\label{norm1}
   \left| \sum_{k=1}^{\infty} x_k  \phi_k \right| \le  1 , 
\end{equation}   
where $\phi_k = \frac{\alpha_k}{w_k}$ $(k \in \bN)$. 
Given $N \in \bN$, define $x \in I_w$ by $x_k = \frac{w_k}{C |w_k|}$ for $k \le N$ and $x_k = 0$ otherwise.
Then we have $\|x\|\leq  C \, \|x\|_{\infty} = 1$. Inserting this $x$ in the inequality (\ref{norm1}) yields 
$\sum_{k=1}^{N} |\phi_k| \le  C$. Since this holds for any $N \in \bN$, we conclude that $\phi \in l^1$ and $\|\phi\|_1 \le C$. 

It remains to show that the inequality (\ref{norm1}) holds for any $x \in B$. 
We may assume that $x \ne 0$, and so $P_N x \ne 0$ for all $N \in \bN$ large enough. 
Since for such an $N$ the vector $\frac{P_N x}{\|P_N x\|}$ belongs to $B \cap I_w$, we have 
$$  \left| \sum_{k=1}^{N} x_k  \phi_k \right| \le \|P_N x\| , $$ 
and so 
$$  \left| \sum_{k=1}^{\infty} x_k  \phi_k \right| \le  \liminf_{N \to \infty} \|P_N x \| = \|x\| \leq 1 . $$
This shows that  $\| \phi \|_* \le 1$. On the other hand, since $\sum_{k=1}^{\infty} w_k  \phi_k = \sum_{k=1}^{\infty} \alpha_k  = 1$ and 
$\| w \| = 1$, we have $\| \phi \|_* \ge 1$. This completes the proof.
\qed
\end{proof}

\vspace{5mm}
\section{Examples}

Theorem \ref{Zenger} raises some questions. Two of them are answered by the following examples.
The first example shows that in Theorem \ref{Zenger} the assumptions (\ref{P_N}) cannot be removed.

\begin{example}
The norm 
$$ \|x\| = \|x\|_{\infty} + \limsup_{n \rightarrow \infty} |x_n| $$ 
is equivalent to the original norm $\|\cdot\|_{\infty}$ on $l^\infty$, as 
$\|x\|_{\infty}  \leq \|x\| \le 2 \, \|x\|_{\infty}$.  
For every $x \in l^\infty$ and every $N \in \bN$, we have 
$\|P_N x\| = \|x\|_{\infty} \leq \|x\|$, and so $\|P_N\| \leq 1$ and 
$\limsup_{N \to \infty} \|P_N x \| \leq \|x\|$.
If $u = (1, 0, 0, 0, \ldots) \in l^\infty$ then $\|u\| = \|P_N u\| = 1$, and thus  $\|P_N\| = 1$.
However, for $e = (1, 1, 1, \ldots) \in l^\infty$ we have  $\|P_N e\| = 1$ and $\|e\| = 2$, 
so that the right-hand assumption in (\ref{P_N}) is not satisfied. 

Let $\alpha = (\alpha_1, \alpha_2, \ldots) \in l^1$ be a sequence of strictly positive numbers such that 
$\sum_{k=1}^{\infty} \alpha_k = 1$. 
To prove that the conclusion of  Theorem \ref{Zenger} does not hold, 
pick any $w \in l^\infty$ such that $\| w \| = 1$  and $w_k \neq 0$ for all $k \in \mathbb{N}$.
Since $\|w\|_{\infty} \leq \|w\| = 1$, we have $|w_k| \leq 1$ for all $k \in \mathbb{N}$ and 
$|w_n| < 1$ for some positive integer $n$.
Therefore, it follows from 
$$ \sum_{k=1}^{\infty} \frac{\alpha_k}{|w_k|} > \sum_{k=1}^{\infty} \alpha_k = 1 $$
that there is $N \in \bN$ such that 
$$ \sum_{k=1}^{N} \frac{\alpha_k}{|w_k|} > 1 . $$
Define the vector $\phi = (\phi_1, \phi_2, \ldots)$ by $\phi_k = \frac{\alpha_k}{w_k}$ for $k \in \bN$, and the vector $x \in l^\infty$ by $x_k = \frac{w_k}{|w_k|}$ for $k \le N$ 
and $x_k = 0$ otherwise. Then $\|x\| = 1$ and 
$$ \sum_{k=1}^{\infty} x_k  \phi_k = \sum_{k=1}^{N} \frac{\alpha_k}{|w_k|} > 1 . $$
This implies that $\| \phi \|_* > 1$, and so the conclusion of Theorem \ref{Zenger} is not valid. 
\end{example}

Examining the proof of Theorem \ref{Zenger} one may ask whether the vector $|w|$ is necessarily the order unit of $l^\infty$, i.e., 
the order ideal $I_w$ generated by $w$ is equal to $l^\infty$. The following example shows that this is not the case.

\begin{example}
Set $w = (1,  2^{-1}, 2^{-2}, 2^{-3}, \ldots) \in l^\infty$, and define the norm on $l^\infty$ by
$$ \|x\| = \|x\|_{\infty} + \|x - x_1 w\|_{\infty} . $$
Since 
$$ \|x\|_{\infty} \leq \|x\| \leq \|x\|_{\infty} + ( \|x\|_{\infty} + |x_1| \|w\|_{\infty}) \leq 3 \, \|x\|_{\infty} , $$
this norm is equivalent to the original norm on $l^\infty$.
Clearly, $\|w\| =  \|w\|_{\infty} = 1$ and $I_w \neq l^\infty$. 
Since 
$$ \|P_N x\| = \max_{1 \le k \le N} |x_k| +  \max \left\{ \max_{2 \le k \le N} \left| x_k - \frac{x_1}{2^{k-1}} \right|, 
\frac{|x_1|}{2^{N}} \right\}  $$
for every $x \in l^\infty$ and every $N \in \bN$, we have 
$$ \lim_{N \to \infty} \|P_N x \| = \|x\|_{\infty} + \|x - x_1 w\|_{\infty} = \|x\| $$
and 
$$ \|P_N x\| \leq \|x\|_{\infty} + \|x - x_1 w\|_{\infty} + \frac{|x_1|}{2^{N}} \leq \|x\| \left( 1 + \frac{1}{2^{N}} \right) , $$
and so $\lim_{N \to \infty} \|P_N \| = 1$. Therefore, the norm  $\|\cdot\|$ satisfies all assumptions of Theorem \ref{Zenger}.

Define the vectors $\alpha = (\alpha_1, \alpha_2, \ldots) \in l^1$ and $\phi = (\phi_1, \phi_2, \ldots) \in l^1$ by 
$$ \alpha_k = \frac{3}{4^k} \ \ \ \textrm{and} \ \ \phi_k = \frac{\alpha_k}{w_k} = \frac{3}{2^{k+1}} \ (k \in \bN) . $$
Then $\sum_{k=1}^{\infty} w_k  \phi_k  = \sum_{k=1}^{\infty} \alpha_k = 1$, so that $\| \phi \|_* \geq  1$. 

In order to show that  $\|\phi\|_* = 1$, choose any $x \in l^\infty$ with $\|x\| \le 1$.
Denoting $|x_1| = t \in [0,1]$, we have, for all $k \geq 2$,  
$$ 1 \geq \|x\|_{\infty} + \|x - x_1 w\|_{\infty} \geq |x_1| + \left| x_k - \frac{x_1}{2^{k-1}} \right| \geq
t + |x_k| - \frac{t}{2^{k-1}}, $$
and so 
$$ |x_k| \leq 1 - t + \frac{t}{2^{k-1}}  . $$
It follows that  
$$ \left| \sum_{k=1}^{\infty} x_k  \phi_k \right| \leq 3 \sum_{k=1}^{\infty} \frac{|x_k|}{2^{k+1}} \leq 
\frac{3 t}{4} + 3 \sum_{k=2}^{\infty} \left( \frac{1 - t}{2^{k+1}} + \frac{t}{4^{k}} \right) = $$
$$ = \frac{3 t}{4} + \frac{3 (1-t)}{4} + \frac{t}{4} = \frac{3 + t}{4} \leq  1 . $$
This shows that $\| \phi \|_* \leq 1$, as desired.
\end{example}

\vspace{5mm}
\section{Possible application in the Arrow-Debreu model}

In the well-known Arrow-Debreu model from the mathematical economy  (see \cite{AlBrBu} or \cite{AlBu03}), 
various commodities are exchanged, produced and consumed.
The classical model deals with finitely many commodities, but already Debreu proposed to study commodity spaces with an infinite number of commodities.
Let us suppose that there are countably many commodities, and that the commodity space is a subset $X$ of the real vector space $l^\infty$. 
Each vector $x = (x_1, x_2, x_3, \ldots) \in X$ represents a \textit{commodity bundle}, that is, the number $x_k$ is the amount of the $k$-th commodity.
Inputs for production are negatively signed, outputs are positively signed. 
If $p_k$ is the price for one unit of the $k$-th commodity, then we can introduce the \textit{price vector} $p = (p_1, p_2, p_3, \ldots)$, and define 
the \textit{value} of a commodity bundle $x$ at prices $p$ by $\sum_{k=1}^{\infty} x_k  p_k$. The price $p_k$ is negative in the case when
the $k$-th commodity is noxious, and so the consumer pays to get rid of it.
Hence, each price vector defines a linear functional on $l^\infty$, if we define the price space as the real vector space $l^1$.

In this theory every consumer has individual taste or \textit{preference} that is a binary relation $\succeq$  defined on 
the commodity space $X$ which is reflexive, transitive and total (see \cite{AlBrBu} or \cite{AlBu03}). 
Such a preference is usually represented by a \textit{utility function} $u : X \to \bR$ in the following way: 
$$ x \succeq y  \iff u(x) \geq u(y) . $$
An important example of utility functions is a \textit{Cobb-Douglas utility function} defined on the positive cone $\bR^n_+$ of $\bR^n$ by
$$ u(x_1, x_2, \ldots, x_n) = x_1^{\alpha_1} x_2^{\alpha_2} \cdots x_n^{\alpha_n} , $$
where $\alpha_k > 0$ for all $k$ and $\sum_{k=1}^{n} \alpha_k = 1$.
Since composing a utility function by a strictly increasing function does not change the preference relation, 
we can replace the Cobb-Douglas utility function $u$ by the utility function 
$$ \log u(x_1, x_2, \ldots, x_n) = \sum_{k=1}^n \alpha_k \cdot \log(x_k) . $$
Hence, in our commodity space $X \subseteq l^\infty$ it is natural to take the utility function 
$$ F(x_1, x_2, x_3, \ldots ) = \sum_{k=1}^{\infty} \alpha_k \cdot \log |x_k|  , $$
where $\alpha_k > 0$ for all $k$ and $\sum_{k=1}^{\infty} \alpha_k = 1$. This is exactly the function defined in the proof of Theorem \ref{Zenger}.
How can we interpret the conclusions of Theorem \ref{Zenger}?

We consider the preference relation represented by the utility function $F$ that is defined on the unit ball 
$B = \{ x \in l^\infty : \|x\| \le 1 \}$, where $\|\cdot\|$ is an equivalent norm on $l^\infty$ having the properties (\ref{P_N}). 
Theorem \ref{Zenger} gives the commodity bundle $w \in B$ which maximizes the utility function $F$ on $B$, 
so that $w$ is the most desirable bundle in $B$. At the same time, the result determines the price vector 
$\phi = (\phi_1, \phi_2, \phi_3, \ldots)$ with the property that the value $\sum_{k=1}^{\infty} x_k  \phi_k$ of $x \in B$ at prices $\phi$ 
is maximal for the bundle $x = w$. 

\vspace{5mm}
{\bf
\begin{center}
 Acknowledgment
\end{center}
} 
The author was supported in part by the Slovenian Research Agency.

\end{document}